\documentclass[11pt,reqno]{amsart}

\usepackage{mathrsfs}
\usepackage{amsmath}
\usepackage{amssymb, amsmath}
\usepackage{graphicx}
\usepackage{color}
\usepackage{enumerate}
\usepackage{amsmath,bm}
\numberwithin{equation}{section}
\newtheorem{theorem}{Theorem}[section]
\newtheorem{proposition}[theorem]{Proposition}
\newtheorem{lemma}[theorem]{Lemma}

\newcommand{\D}{\displaystyle}
\theoremstyle{definition}

\def\XXint#1#2#3{{\setbox0=\hbox{$#1{#2#3}{\int}$}
     \vcenter{\hbox{$#2#3$}}\kern-.5\wd0}}

\begin{document}
\title[Uniqueness of a non-local semilinear elliptic equation]{Uniqueness and convergence on equilibria of the Keller-Segel system with subcritical mass}
\author{Jun Wang}
\address{ Jun ~Wang,~Faculty of Science, Jiangsu University, Zhenjiang, Jiangsu, 212013, P.R. China}
\email{wangmath2011@126.com}
\author{Zhi-An Wang}
\address{ Zhi-An Wang,~Department of Applied Mathematics, Hong Kong Polytechnic University, Hung Hom, Kowloon, Hong Kong}
\email{mawza@polyu.edu.hk}
\author{ Wen Yang}
\address{ Wen ~Yang,~Wuhan Institute of Physics and Mathematics, Chinese Academy of Sciences, P.O. Box 71010, Wuhan 430071, P. R. China}
\email{wyang@wipm.ac.cn}

\date{}
\begin{abstract}
This paper is concerned with the uniqueness of solutions to the following nonlocal semi-linear elliptic equation
\begin{equation}\label{ellip}\tag{$\ast$}
\Delta u-\beta u+\lambda\frac{e^u}{\int_{\Omega}e^u}=0~\mathrm{in}~\Omega,
\end{equation}
where $\Omega$ is a bounded domain in $\mathbb{R}^2$ and $\beta, \lambda$ are positive parameters. The above equation arises as the stationary problem of the well-known classical Keller-Segel model describing chemotaxis. \textcolor{black}{For equation \eqref{ellip} with Neumann boundary condition, we establish an integral inequality and prove that the solution of (\ref{ellip})  is unique if $0<\lambda \leq 8\pi$ and $u$ satisfies some symmetric properties. While for \eqref{ellip} with Dirichlet boundary condition, the same uniqueness result is obtained without symmetric condition by a different approach inspired by some recent works \cite{gm0,gm}.} As an application of \textcolor{black}{the uniqueness} results, we prove that the radially symmetric solution of the classical Keller-Segel system with subcritical mass subject to Neumann boundary conditions will converge to the unique constant equilibrium as time tends to infinity if $\Omega$ is a disc in two dimensions. As far as we know, this is the first result that asserts the exact asymptotic behavior of solutions to the classical Keller-Segel system with subcritical mass in two dimensions.
\end{abstract}
\maketitle

\section{Introduction}
This paper is concerned with a system of partial differential equations modeling chemotaxis which refers to the active motion of biological species towards higher concentrations of chemical substances that they emit themselves. The classical chemotaxis model is well-known as the Keller-Segel (KS) system \cite{ks, ks1}, reading as
\begin{equation}
\label{1.1}
\begin{cases}
v_t=\nabla\cdot(\nabla v-v\nabla u)~&\mbox{in}~\Omega,\\
u_t=\Delta u-\beta u+v~&\mbox{in}~\Omega,\\
\partial_{\nu}v=\partial_{\nu}u=0~&\mbox{on}~\partial\Omega,\\
v(x,0)=v_0(x),\quad u(x,0)=u_0(x)~&\mbox{in}~\Omega,
\end{cases}
\end{equation}
where $\Omega$ is smooth bounded domain in $\mathbb{R}^N(N\geq 2)$, $v(x,t)$ and $u(x,t)$ denote the cell density and chemical concentration, respectively;  $\beta$ is a positive constant accounting for the chemical death rate, $\nu$ is the unit outward normal vector at the boundary $\partial \Omega$.

The underlying equations in system \eqref{1.1} were proposed by Keller and Segel in 1970 \cite{ks, ks1} to describe the aggregation phase of cellular slime molds {\it Dictyostelium discoideum} in response to the chemical substance Cyclic adenosine monophosphate (cAMP) that they secreted. An immediate result derived from (\ref{1.1}) is the mass conservation for $v(x,t)$ by integrating the first equation
\begin{equation*}
\lambda=: \int_{\Omega}v(x,t)dx=\int_{\Omega}v_0(x)dx.
\end{equation*}
Among other things, the most striking feature of the KS system (\ref{1.1}) lies in the existence of critical space dimension and critical mass.  Roughly speaking, in one-dimensional space ($N=1$), the KS system (\ref{1.1}) admits globally bounded classical solutions \cite{Hillen-Potapov, Osaki}. In the two dimensional radially symmetric domain (disc),  global bounded classical solutions exist (cf. \cite{nsy}) if $\lambda<8\pi$ (subcritical mass), whereas solutions may blow up in finite or infinite time (cf. \cite{Herrero1, Horstmann-Wang}) if $\lambda>8\pi$ (super-critical mass), where the critical mass $8\pi$  becomes $4\pi$ if $\Omega$ is a general domain without symmetry. In three dimensions, the solution may blow up in finite time for any mass $\lambda>0$ (cf. \cite{Winkler1}). Therefore $N=2$ is a borderline space dimension where a critical mass exists, and as such the KS model (\ref{1.1}) and its variants (extensions) have attracted extensive attentions and a vast number of fruitful results have been obtained (cf. review papers \cite{Bellomo, H-review, HP-review} and book \cite{Perthame-book}). However, as far as we know,  the  long-time behavior of solutions to the Keller-Segel system (\ref{1.1}) with subcritical mass in two dimensions still remains unknown. It is the purpose of this paper to explore this open question. Precisely we shall show that the radial solution of (\ref{1.1}) in a disc with subcritical mass will converge to the unique constant equilibrium as time tends to infinity. Since it has been shown in \cite{fp} that the globally bounded solution (if it exists) of (\ref{1.1}) converges to the steady states in $L^\infty$-norm, our question boils down to prove the uniqueness of constant equilibrium for the stationary problem of (\ref{1.1}) in $\Omega \subset \mathbb{R}^2$:
\begin{equation}\label{1.1s}
\begin{cases}
\nabla\cdot(\nabla v-v\nabla u)=0~&\mbox{in}~\Omega,\\
\Delta u-\beta u+v=0~&\mbox{in}~\Omega,\\
\partial_{\nu}v=\partial_{\nu}u=0~&\mbox{on}~\partial\Omega,
\end{cases}
\end{equation}
in the case of subcritical mass $\int_\Omega v(x)dx<8\pi$.
To study the stationary system \eqref{1.1s}, we note that the first equation can be written as
\begin{equation*}
\nabla\cdot(v\nabla(\log v- u))=0.
\end{equation*}
Testing the above equation against $\log v-u$, then an integration by parts shows that any solution of \eqref{1.1s} verifies the equation
\begin{equation*}
\begin{aligned}
\int_{\Omega} v|\nabla(\log v-u)|^2dx=0,
\end{aligned}
\end{equation*}
so that $v=Ce^{u}$ for some positive constant $C$. Hence, all the solutions $(v,u)$ of \eqref{1.1s} would satisfy the relation $v=Ce^{u}$. Denoting
$$\lambda =\int_\Omega v(x)dx,$$
we have $C=\frac{\lambda}{\int_{\Omega}e^{u}dx}.$ Therefore, the stationary problem (\ref{1.1s}) is equivalent  to the following nonlocal semi-linear Neumman problem
\begin{equation}
\label{1.6}
\begin{cases}
\Delta u-\beta u+\D\lambda\frac{e^u}{\int_{\Omega}e^{u}dx}=0~&\mathrm{in}~\Omega,\\
\partial_{\nu}u=0~&\mathrm{on}~\partial\Omega,
\end{cases}
\end{equation}
and $v=\frac{e^u}{\int_\Omega e^udx}$.
Integrating the first equation of (\ref{1.6}), one immediately obtains that
$$\int_\Omega u(x)dx=\frac{\lambda}{\beta}.$$
Then by a (new) shifted variable
\begin{equation}
\label{1.uv}
U(x)=u(x)-\bar u, \quad \bar u=\frac{1}{|\Omega|}\int_{\Omega}u(x)dx=\frac{\lambda}{\beta |\Omega|},
\end{equation}
the problem (\ref{1.6}) can be transformed as the following

\begin{equation}
\label{1.5}
\begin{cases}
\Delta U-\beta U+\lambda\Big(\frac{e^{U}}{\int_{\Omega}e^{U}dx}-\frac{1}{|\Omega|}\Big)=0~&\mathrm{in}~\Omega,\\
\partial_{\nu}U=0~&\mathrm{on}~\partial\Omega.
\end{cases}
\end{equation}
From \eqref{1.uv} we know that the mean of $U$ is zero, namely
\begin{equation}
\label{1.int-v}
\int_{\Omega}U(x)dx=0.
\end{equation}
When $\beta=0$, the nonlinear differential equation in \eqref{1.5} is closely related to the Gaussian curvature problem on a surface (see \cite{m}). In Onsager's vortex theory the asymptotic limit of the Gibbs measure yields to similar problems (see \cite{clmp,ck,k}). Moreover the first equation of \eqref{1.5} on a torus arises in the Chern-Simons gauge theory \cite{t} and has been investigated among others by Struwe-Tarantello \cite{st}. Chen-Lin \cite{cl} and Machioldi \cite{m1} have independently derived the Leray-Schauder Topological degree for \eqref{1.5} on the Riemann surface without boundary.
By assuming $\beta>\frac{\lambda}{|\Omega|}-\lambda_1$ and $\lambda>4\pi,$ where $\lambda_1$ is referred to the first eigenvalue of the Neumann eigenvalue problem, Wang-Wei \cite{ww} and Horstmann \cite{h} have independently shown the existence of non-constant solutions for \eqref{1.5}. Very recently, Battaglia obtains the existence of non-constant solutions to \eqref{1.5} with $(\lambda,\beta)$ in a wider range, see \cite[Theorem 1.1]{bluca} for the details.

Since the Neumann problem \eqref{1.5}-\eqref{1.int-v} admits a trivial solution $U=0$, it is natural to ask whether there is any other solution. When $\beta=0$, a simple application of the maximum principle will show that $U\equiv 0$ is the unique solution to problem \eqref{1.5}-\eqref{1.int-v} provided that $\lambda\leq 0$ though it is not of interest in applications. For $\lambda>0$, the first equation of \eqref{1.5} (without Neumann boundary condition) on a standard sphere admits only the constant solution whenever $\lambda<8\pi$ (cf. \cite{ck,cslin,o}), while the uniqueness result also holds for some flat torus provided $\lambda\leq8\pi$ (cf. \cite{ll}). When $\Omega$ is the unit disc, Horstmann-Lucia \cite{hl} proved that $U\equiv0$ is the unique solution for Neumann problem \eqref{1.5}-\eqref{1.int-v} if  $\lambda\leq\frac{32}{\pi}$. The uniqueness result also holds when the solution is constant on the boundary (i.e. $\mathrm{osc}_{\partial\Omega}(u)\equiv0$) and $\lambda\leq8\pi$. As a consequence of their results, we may \textcolor{black}{expect to} conclude that $U\equiv0$ is the unique radially symmetric solution to the Neumann problem \eqref{1.5}-\eqref{1.int-v} if $\lambda\leq8\pi$ and \textcolor{black}{$\Omega$ is the unit disc.}

We remark that the above results for the Neumann problem (\ref{1.5}) with $\beta=0$ can not be converted to the Neumann problem (\ref{1.1s}) since the transformation (\ref{1.uv}) no longer works for $\beta=0$.
Indeed, if $\beta=0$, the integration of the second equation of (\ref{1.1s}) along with the Neumman boundary condition yields that $\int_\Omega vdx=0$, which \textcolor{black}{along with the fact $v\geq 0$} indicates that (\ref{1.1s}) with $\beta=0$ has no solution if $\lambda>0$ \textcolor{black}{or} $u\equiv0$ if $\lambda=0$. Hence the solution for the case $\beta=0$ is clear but not of interest. To the best of our knowledge, there are very few results available for the case $\beta>0$ for which the shifted problem (\ref{1.5}) is equivalent to (\ref{1.6}) under the shifting (\ref{1.uv}). In this paper, we shall investigate the uniqueness of solutions to \eqref{1.6} with $\beta>0$ which is equivalent to (\ref{1.1s}) with $v=\frac{e^u}{\int_\Omega e^udx}$. Hereafter, we shall assume $\beta>0$ unless otherwise stated. When $\Omega$ is the unit disc, Senba-Suzuki \cite[Theorem 4]{ss} have obtained the existence of nontrivial radial solutions of \eqref{1.6} for $\lambda>8\pi$. In this paper, we shall prove that the constant equilibrium $u=\frac{\lambda}{\beta\pi}$ is the only radially symmetric solution to the Neumann problem \eqref{1.6} if $\Omega$ is the unit disc and $\lambda\leq 8\pi$.
Using the maximum principle, one can get $u>0$ in $\Omega$ and we leave the proof in the appendix (see Lemma \ref{le6.1}). Therefore our study will be focused on the positive solutions of \eqref{1.6}.\\

Our first result concerning the Neumann problem \eqref{1.6} is the following:
\begin{theorem}
\label{th1.1}
Let $\Omega$ be a non-empty bounded open set in $\mathbb{R}^2$. Then for $\lambda\leq 8\pi$ the constant $u=\frac{\lambda}{\beta|\Omega|}$ is the unique solution to the problem \eqref{1.6} with $\mathrm{osc}_{\partial\Omega}(u)\equiv0$.
\end{theorem}

\noindent {\bf Remark 1.1:} When $\Omega$ is the unit disc, we can obtain from Theorem \ref{th1.1} that $u=\frac{\lambda}{\beta\pi}$ is the only radially symmetric solution to the problem \eqref{1.6} provided $\lambda\leq 8\pi$.
\medskip


\textcolor{black}{Inspired by the result of Theorem \ref{th1.1}, one may ask what happens if the condition $\mathrm{osc}_{\partial\Omega}u\equiv 0$ is replaced by some symmetric condition, i.e., the solutions which are invariant under the group of isometries of a unit disc. To state the results, we introduce the following classes of functions:}
\begin{equation*}
H^G:=\{w\in H^1(B):w=w\circ g,~\forall g\in G\}.
\end{equation*}
and
\begin{equation*}
\mathring{H}^G:=\{w\in H^G:\int_\Omega w=0\},
\end{equation*}
where $B$ denotes the unit disc. For any $\theta\in(0,2\pi)$ and $C\neq 0$ we use the notations
\begin{equation*}
R_{\theta}:=~\mbox{rotation of angle}~\theta,
\end{equation*}
\begin{equation*}
D_C:=~\mbox{reflection with respect to}~\overrightarrow{OC}~(O~\mbox{stands for the orgin point}),
\end{equation*}
and $\langle g\rangle$ stands for the subgroup generated by an isometry $g$. We list the following examples which are used usually:
\begin{enumerate}
  \item [(a)] When $G=SO(2)=\left\{\left(\begin{matrix}\cos\theta&\sin\theta\\-\sin\theta&\cos\theta\end{matrix}\right),\theta\in[0,2\pi)\right\}$, the space $H^G$ consists of the radial $H^1(B)$-functions.
  \item [(b)] $H^{\langle R_{\pi}\rangle}$ consists of the $H^1(B)$-functions satisfying $w(x)=w(-x)$.
  \item [(c)] Let $D_m:\langle R_{2\pi/m},D_C\rangle$ be a group generated by the rotation $R_{2\pi/m}$ and a reflection $D_C$. The group with $2m$ elements is called the $m-$th {\em dihedral} group. For $m\geq3$, it is the symmetry group of regular $m$-polygon.
\end{enumerate}

For  solutions of \eqref{1.6} which are invariant under rotations, we have:
\begin{theorem}
\label{th.rotation}
Let $G=\langle R_{2\pi/m}\rangle,~m\in\mathbb{N}$ with $m\geq 2$. If a non-constant function $u\in H^G$ solves the problem \eqref{1.6} in $B$, then
\begin{equation}
\label{1.12}
\lambda>\Lambda_m=\begin{cases}
\frac{64}{\pi},~&\mbox{if}~m=2,\\
8\pi,~&\mbox{if}~m\geq3.
\end{cases}
\end{equation}
\end{theorem}

\medskip

To prove Theorems \ref{th1.1}-\ref{th.rotation}, we first consider a more general problem
\begin{equation}
\label{1.fg}
\begin{cases}
\Delta u-g(u)+f(u)=0~&\mathrm{in}~\Omega,\\
u>0~&\mathrm{in}~\Omega,\\
\partial_{\nu}u=0~&\mathrm{on}~\partial\Omega,
\end{cases}
\end{equation}
where $f,g$ satisfy the following conditions:
\begin{equation}
\label{1.fga}
f,g\in C^2(\mathbb{R}),\quad f(t),f'(t)>0,~g(t),g'(t)\geq0~\mbox{for}~t\geq0.
\end{equation}
For equation \eqref{1.fg} with $f,g$ satisfying \eqref{1.fga}, we derive an integral inequality
\begin{equation}
\label{1.ine}
\frac{A_g}{2}\int_{-\infty}^{\infty}f'(t)\mu(t)(|\Omega|-\mu(t))dt\geq
\int_{-\infty}^{\infty}f(t)I_{\Omega}^2(\mu(t))dt,
\end{equation}
where $$A_g=\frac{1}{|\Omega|}\int_{\Omega}g(u)dx,\quad \mu(t)=|\{x\in\Omega:u(x)>t\}|$$
and $I_{\Omega}(s)$ is referred to the ``isoperimetric profile" of $\Omega$ with volume $s$, the detailed definition of the ``isoperimetric profile" will be given in Section 2. \textcolor{black}{In \cite{hl,l}, the authors study the problem \eqref{1.fg} with $g(u)$ replaced by a constant $A=\frac{1}{|\Omega|}\int_{\Omega}f(u)dx$. In their approach, a similar inequality of \eqref{1.ine} is obtained with $A_g$ replaced by $A$. Different from their problem, we assume that $g(u)$ is a non-constant function with some increasing property. To derive the inequality \eqref{1.ine}, we have to estimate the integration of $g(u)$ with respect to the level set of $u$. With a simple manipulation, see Lemma \ref{le2.2}, we manage to estimate the integration of $g(u)$ in terms of the mean value $A_g$. This is the key point which helps us to generalize all the results to $g(u)$ - a nontrivial function. Another new ingredient in our proof is that we find to study the augmented  functions $\Psi$ and $\tilde\Psi$ (see the definition of $\Psi$ and $\tilde\Psi$ in \eqref{3.psi}) together to show that the jump of each discontinuous point is positive, see \eqref{3.2}. It is different from the simpler case of constant function $g(u)$, where $\Psi$ and $\tilde\Psi$ \textcolor{black}{can be} studied separately (cf. \cite{hl,l}).}
\medskip

We shall apply our uniqueness results to exploit the asymptotic behavior of solutions to the Keller-Segel system \eqref{1.1}. Specifically we show that if $\lambda<\Lambda_m$ any rotationally invariant solution to \eqref{1.1} is uniformly bounded, exists globally and converges to the unique constant equilibrium.
\begin{theorem}
\label{th1.4}
Consider the problem \eqref{1.1} in the unit disc $B\subset\mathbb{R}^2$. Let $(v,u)$ be a $C^2$-solution to \eqref{1.1} belonging to $H^G\times\mathring{H}^G$ with $G=\langle R_{2\pi/m}\rangle~(m\geq2)$. If $\lambda<\Lambda_m$ with $\Lambda_m$ defined in \eqref{1.12}, then the solution of \eqref{1.1} is globally defined and
\begin{equation}\label{asy}
\lim_{t\to\infty}(v,u)(\cdot,t)=\left(\frac{\lambda}{\pi},\frac{\lambda}{\beta\pi}\right)\ \ \mbox{in}\ \ C^1(\overline B).
\end{equation}
In particular, if $\lambda<8\pi$, any radial solution $(v,u)$ of \eqref{1.1} will satisfiy \eqref{asy}.
\end{theorem}

We remark that the convergence of solutions \eqref{1.1} to the constant equilibrium as time tends to infinity in the critical case $\lambda=8 \pi$ remains unknown. Indeed the asymptotic behavior of solutions to the classical Keller-Segel system in the case of critical mass is a rather complicated issue and the answer was only partially given in the whole space $\mathbb{R}^2$ for the case $\beta=0$ (cf. \cite{Lopez1, Lopez2}). Although the asymptotics of solutions for critical mass in a bounded domain still remains unknown in this paper, our results in Theorem \ref{th1.1} and Theorem \ref{th.rotation} will shed lights on this problem for further pursues in the future. \\


The last problem to be considered in this paper is the uniqueness of solutions for the following Dirichlet problem:
\begin{equation}
\label{1.dirichlet}
\begin{cases}
\Delta u-\beta u+\lambda\frac{e^{u}}{\int_{\Omega}e^{u}dx}=0~&\mbox{in}~\Omega,\\
u=0~&\mbox{on}~\partial\Omega.
\end{cases}
\end{equation}
The above problem is the steady state problem of the Keller-Segel system with mixed zero-flux and Dirichelt boundary conditions
\begin{equation*}
\label{1.d.system}
\begin{cases}
v_t=\nabla\cdot(\nabla v-v\nabla u)~&\mbox{in}~\Omega,\\
u_t=\Delta u-\beta u+v~&\mbox{in}~\Omega,\\
\partial_{\nu}v-v\partial_{\nu}u=u=0~&\mbox{on}~\partial\Omega,\\
u(x,0)=u_0(x),\quad v(x,0)=v_0(x)~&\mbox{in}~\Omega.
\end{cases}
\end{equation*}
When $\beta=0$, Suzuki \cite{s} proved that if $\Omega$ is simply-connected, then the problem \eqref{1.dirichlet} has a unique solution for $0<\lambda<8\pi$ . The uniqueness result for $\lambda=8\pi$ is obtained by Chang-Chen-Lin \cite{ccl}. Later in \cite{bl1} Bartolucci and Lin extended the result to multiply-connected domains. Recently, based on the Bol's inequalities and equi-measurable symmetric rearrangement, Gui-Moradifam \cite{gm0} developed a new tool named ``Sphere Covering Inequality". This inequality and its generalizations are applied to establish the best constant in a Moser-Trudinger type inequalities, some symmetry and uniqueness results for the mean field equations and Onsager vortex (cf. \cite{gjm,gm0,gm1,gm,sstw} for details). Based on their results in \cite{gm0,gm}, we shall derive the uniqueness of \eqref{1.dirichlet} in the subcritical mass cases.

\begin{theorem}
\label{th.dirichlet}
Let $\Omega$ be an open bounded and simply-connected set in $\mathbb{R}^2$. If $0\leq\lambda<8\pi,$ then there exists a unique solution of the problem \eqref{1.dirichlet} with $\beta>0$. While if $\lambda=8\pi$, equation \eqref{1.dirichlet} has at most one solution.
\end{theorem}

\noindent {\bf Remark 1.2:} \textcolor{black}{We remark that the the Dirichlet problem \eqref{1.dirichlet} no longer has a constant solution provided that $\lambda>0$. In other words,  the unique solution of the problem \eqref{1.dirichlet} for $0\leq \lambda <8\pi$ must be non-trivial, which differs from the Neumann problem \eqref{1.6}.
Furthermore, we can show that the degree of equation \eqref{1.dirichlet} is $0$ for $\lambda\in (8\pi,16\pi)$, see Theorem \ref{th5.2}. As a consequence, the Dirichlet problem \eqref{1.dirichlet}  has no solution or at least two solutions if $\lambda\in(8\pi,16\pi)$.
On the other hand, we note that the existence of solutions to \eqref{1.dirichlet} with $\lambda=8\pi$ is still unknown and it may depend on the topology of $\Omega$ (cf. \cite{bl1} for $\beta=0$). In this sense, $\lambda=8\pi$ is a threshold for the uniqueness of solutions for the Dirichlet problem \eqref{1.dirichlet}.}
\medskip


The paper is organized as follows. In Section 2, we derive a differential inequality which involves the distribution of $u$, the function $\int_{\{u>t\}}f(u)dx$, the average of function $g(u)$ and the isoperimetric profile of the domain. Based on this result, in Section 3 we derive an integral inequality which is the key to the proof of Theorems \ref{th1.1}-\ref{th.rotation}. The Section 4 is devoted to the proof of our results on the Neumann problem \eqref{1.6}. While in Section 5 we study the uniqueness of solutions to the Dirichlet problem. In the appendix, we give a proof for the positivity of solutions to \eqref{1.6} in $\Omega.$

\medskip
\section{A differential inequality}
In the present section, we shall establish a differential inequality, which plays an important role in our discussion. Given two functions $f,g$ which satisfy \eqref{1.fga}. We set
\begin{equation}
\label{2.2}
\mathcal{C}(f,g)=\{t\mid f(t)=g(t),~t\geq0\},
\end{equation}
and consider the following nonlinear problem:
\begin{equation}
\label{2.3}
\begin{cases}
\Delta u-g(u)+f(u)=0~&\mathrm{in}~\Omega,\\
u>0~&\mathrm{in}~\Omega,\\
\partial_{\nu}u=0~&\mathrm{on}~\partial\Omega.
\end{cases}
\end{equation}

To proceed with our argument, we make the following preparation. Denote by $\mathcal{H}^s$ the $s$-dimensional Hausdorff measure in $\Omega$. Given $\omega\subset\Omega$, its perimeter relative to $\Omega$ is defined as
\begin{equation*}
\mathcal{P}(\omega,\Omega):=\mathcal{H}^1(\partial\omega\cap \Omega),
\end{equation*}
and its area $\mathcal{H}^2(\omega)$ will be denoted by $|\omega|.$ (See Figure 1 for an illustration of $\partial\omega\cap\Omega$.)
\medskip

\begin{figure}[h]
\includegraphics[width=12cm]{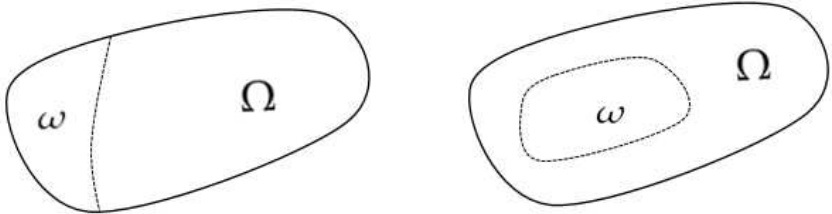}
\caption{For the figure on the
left, the part on the left of the dashed curve is $\omega$ and the
dashed curve represents $\partial\omega\cap\Omega$, for the figure
on the right, the part enclosed by the dashed closed curve is
$\omega$ and the dashed closed curve represents
$\partial\omega\cap\Omega.$}
\end{figure}

\noindent{\bf Definition 2.1.} Let $\mathcal{O}_{\Omega}$ be the class of open subsets $\omega\subset\Omega$ satisifying
\begin{equation*}
\omega=\omega'\cap\Omega,~\omega'\Subset\mathbb{R}^2~\mathrm{of~class}~C^1.
\end{equation*}
The ``isoperimetric profile" of $\Omega$ is the function $I_{\Omega}:[0,|\Omega|]\to(0,\infty)$ defined as
\begin{equation*}
I_{\Omega}(s):=\inf\{\mathcal{H}^1(\partial\omega\cap\Omega):\omega\in\mathcal{O}_{\Omega},\mathcal{H}^2(\omega)=s\},~\forall s\in(0,|\Omega|],
\end{equation*}
and we set $I_\Omega(0)=0.$
\medskip

\noindent{\bf Remark 2.1:} For each $\omega\in\mathcal{O}_\Omega$ the boundary of $\omega$ in $\Omega$ is a $1$-submanifold of class $C^1$. We remark that only $\partial\omega\cap\Omega$ is taken into consideration in the definition of the isoperimetric profile. We mention two properties of the isoperimetric profile that will be used in the following:
\begin{equation}
\label{2.4}
I_{\Omega}(s)=I_{\Omega}(|\Omega|-s),~\forall s\in[0,|\Omega|];
\end{equation}
\begin{equation}
\label{2.5}
I_{\Omega}(s)=0\Leftrightarrow s=0~\mathrm{or}~s=|\Omega|.
\end{equation}
The symmetric property \eqref{2.4} readily follows from the definition of isoperimetric profile, while for \eqref{2.5} we refer to \cite{g}.

We will need the following lemma.

\begin{lemma}
\label{le2.1}
Assume that \eqref{1.fga} holds. Then any non-constant solution $u$ of problem \eqref{2.2} satisfies
\begin{equation*}
\label{2.6}
\mathcal{H}^2(\{u\mid f(u)\neq g(u)\})=0.
\end{equation*}
\end{lemma}

\begin{proof}
Given a fixed $t\in\mathbb{R}\setminus \mathcal{C}(f,g)$, we divide $\{u=t\}:=E_1(t)\cup E_2(t)$ with
\begin{equation*}
E_1(t):=\{x\in\Omega:u(x)=t,\nabla u\neq0\},
\end{equation*}
and
\begin{equation*}
E_2(t):=\{x\in\Omega,u(x)=t,\nabla u(x)=0\}.
\end{equation*}
Using the implicit function theorem, we get the set $E_1(t)$ is locally a one-dimensional manifold. Then we can deduce that $E_1(t)$ is at most countable union of sets of measure zero. Hence $\mathcal{H}^2(E_1(t))=0$ for all $t\in\mathbb{R}$.

For the set $E_2(t)$, using the equation \eqref{2.3} we have
\begin{equation*}
\Delta u\neq0,~\mbox{whenever}~u(x)\notin \mathcal{C}(f,g).
\end{equation*}
Therefore if $t\notin \mathcal{C}(f,g)$, then we have
\begin{equation*}
\begin{aligned}
E_2(t)\subseteq&\{x\in\Omega:\nabla u(x)=0,u(x)\notin \mathcal{C}(f,g)\}\\
\subseteq&\bigcup_{i=1}^2\{x\in\Omega:\partial_iu(x)=0,\nabla(\partial_iu)(x)\neq0\}.
\end{aligned}
\end{equation*}
Thus, for $t\notin \mathcal{C}(f,g)$, the set $E_2(t)$ is contained in a finite union of $1$-submanifolds. So we conclude that $\mathcal{H}^2(E_2(t))=0$ for all $t\in\mathbb{R}\setminus \mathcal{C}(f,g).$

From the above discussion we get $\mathcal{H}^2(E(t))=0$ for all $t\in\mathbb{R}^2\setminus \mathcal{C}(f,g)$, which finishes the proof.
\end{proof}

For any solution $u$ of \eqref{2.2}, we introduce the following notations:
\begin{equation*}
\label{2.7}
F(t):=\int_{\{u>t\}}f(u)dx,~ G(t):=\int_{\{u>t\}}g(u)dx,~ \mu(t):=|\{u>t\}|=\mathcal{H}^2(\{u>t\}),
\end{equation*}
and
\begin{equation*}
\label{2.8}
\tilde F(t):=\int_{\{u<t\}}f(u)dx,~ \tilde G(t):=\int_{\{u<t\}}g(u)dx,~ \tilde \mu(t):=|\{u<t\}|=\mathcal{H}^2(\{u<t\}).
\end{equation*}

By Lemma \ref{le2.1}, the functions $F,\tilde F,G,\tilde{G},\mu,$ and $\tilde{\mu}$ are continuous on $\mathbb{R}\setminus\mathcal{C}(f,g)$. Furthermore, if the set $\mathcal{C}(f,g)$ is finite, then the above functions are monotone and therefore differentiable a.e. $t\in\mathbb{R}.$

We set
\begin{equation*}
\label{2.9}
A_g=\frac{1}{|\Omega|}\int_{\Omega}g(u)dx.
\end{equation*}
Before stating the main result of this section, we give the following lemma.

\begin{lemma}
\label{le2.2}
Suppose that $h$ is a non-decreasing function. For any solution $u$ of \eqref{2.3}, it holds that
\begin{equation*}
A_h|\{u>t\}|\leq\int_{\{u>t\}}h(u)dx~\mathrm{and}~A_h|\{u<t\}|\geq\int_{\{u<t\}}h(u)dx,~\forall t\in[\min_{\Omega}u,\max_{\Omega}u].
\end{equation*}
\end{lemma}

\begin{proof}
We only give the proof of the first inequality, the other one can be proved similarly. Using $h$ is non-decreasing, we immediately get
\begin{equation*}
\frac{\int_{\{u>t\}}h(u)dx}{|\{u>t\}|}\geq\frac{\int_{\{u\leq t\}}h(u)dx}{|\{u\leq t\}|}.
\end{equation*}
As a consequence
\begin{equation*}
\frac{\int_{\{u>t\}}h(u)dx}{|\{u>t\}|}\geq\frac{\int_{\{u>t\}}h(u)dx+\int_{\{u\leq t\}}h(u)dx}{|\{u>t\}|+|\{u\leq t\}|}=A_h,
\end{equation*}
which implies
\begin{equation*}
A_h|\{u>t\}|\leq\int_{\{u>t\}}h(u)dx.
\end{equation*}
It proves the result.
\end{proof}

Now we establish the main result in this section.

\begin{proposition}
\label{pr2.1}
Let $f,g$ be a function satisfying \eqref{1.fga}. Assume $\mathcal{C}(f,g)$ is finite. Then any non-constant solution $u$ of \eqref{2.3} satisfies the following inequalities:
\begin{align}
\label{2.10}
&\frac{d}{dt}\left(A_gf\mu^2-F^2\right)\geq A_gf'(t)\mu^2(t)+2f(t)I_{\Omega}^2(\mu(t)),~\forall t\in\mathbb{R}\setminus\mathcal{D},\\
\label{2.11}
&\frac{d}{dt}\left(A_gf\tilde \mu^2-\tilde F^2\right)\geq A_gf'(t)\tilde \mu^2(t)+2f(t)I_{\Omega}^2(\tilde \mu(t)),~\forall t\in\mathbb{R}\setminus\mathcal{D},
\end{align}
where $\mathcal{D}:=\{u(x):x\in\Omega,\nabla u(x)=0\}$ and $I_{\Omega}$ stands for the isoperimetric profile of $\Omega.$
\end{proposition}

\begin{proof}
At first, we notice that Sard's Theorem ensures that the set of critical value $\mathcal{D}$ associate to $u$ has Lebesgue measure zero in $\mathbb{R}.$

Let us prove \eqref{2.10} first. By Lemma \ref{le2.1}, the functions
$F,$ $G$ and $\mu$ are continuous on $\mathbb{R}\setminus
\mathcal{C}(f,g)$. Therefore, by using the co-area formula, we obtain
\begin{equation}
\label{2.12}
\mu'(t)=-\int_{\{u=t\}}\frac{1}{|\nabla u|}d\mathcal{H}^1,~\forall t\in \mathbb{R}\setminus\mathcal{D},
\end{equation}
\begin{equation}
\label{2.13}
F'(t)=-\int_{\{u=t\}}\frac{f(u)}{|\nabla u|}d\mathcal{H}^1=f(t)\mu'(t),~\forall t\in \mathbb{R}\setminus\mathcal{D},
\end{equation}
\begin{equation}
\label{2.14}
G'(t)=-\int_{\{u=t\}}\frac{g(u)}{|\nabla u|}d\mathcal{H}^1=g(t)\mu'(t),~\forall t\in \mathbb{R}\setminus\mathcal{D}.
\end{equation}
Secondly, by integrating equation \eqref{2.3} on the set $\{u>t\}$ and using the Stoke's Theorem, we obtain
\begin{equation}
\label{2.15}
\int_{\partial \{u>t\}}|\nabla u|d\mathcal{H}^1=F(t)-G(t),~\forall t\in\mathbb{R}\setminus\mathcal{D}.
\end{equation}
Since
\begin{equation*}
\begin{aligned}
\partial\{ u>t\}=~&(\partial\{u>t\}\cap \Omega)\cup (\partial\{u>t\}\cap\partial\Omega)\\
=~&(\{u=t\}\cap\Omega)\cup(\partial\{u>t\}\cap\partial \Omega),
\end{aligned}
\end{equation*}
and furthermore $u$ satisfies the Neumann boundary condition on $\partial\{u>t\}\cap\partial\Omega$, the left hand side of \eqref{2.15} equals to $\int_{\{u=t\}\cap\Omega}|\nabla u|d\mathcal{H}^1.$ Based on this observation, we have
\begin{equation}
\label{2.16}
F(t)=G(t)+\int_{\{u=t\}\cap\Omega}|\nabla u|d\mathcal{H}^1,~\forall t\in\mathbb{R}\setminus\mathcal{D}.
\end{equation}
Using \eqref{2.16} and the assumption $f\geq0$ we derive:
\begin{equation}
\label{2.17}
\begin{aligned}
-F'(t)F(t)=~&\left(G(t)+\int_{\{u=t\}\cap\Omega}|\nabla u|d\mathcal{H}^1\right)\int_{\{u=t\}}\frac{f(t)}{|\nabla u|}d\mathcal{H}^1\\
=~&f(t)\int_{\{u=t\}\cap\Omega}|\nabla u|d\mathcal{H}^1\int_{\{u=t\}}\frac{1}{|\nabla u|}d\mathcal{H}^1+\left(\int_{\{u=t\}}\frac{f(t)}{|\nabla u|}d\mathcal{H}^1\right)\int_{\{u>t\}}g(u)dx\\
\geq~&f(t)\left(\int_{\{u=t\}\cap\Omega}d\mathcal{H}^1\right)^2-f(t)A_g\mu(t)\mu'(t)\\
=~&f(t)[\mathcal{H}^1(\{u=t\}\cap\Omega)]^2-f(t)A_g\mu(t)\mu'(t),
\end{aligned}
\end{equation}
where we have used the Schwarz inequality and Lemma \ref{le2.2}. Recall the definition of isoperimetric profile of $\Omega$, we have
\begin{equation}
\label{2.18}
\int_{\{u=t\}\cap\Omega}d\mathcal{H}^1=\mathcal{P}(\{u>t\},\Omega)\geq I_{\Omega}(\mu(t)).
\end{equation}
Hence, \eqref{2.17} and \eqref{2.18} yield
\begin{equation*}
\label{2.19}
\frac{1}{2}\left(A_gf\mu^2-F^2\right)'(t)\geq f(t)I_{\Omega}^2(\mu(t))+\frac{A_g}{2}f'(t)\mu^2(t),
\end{equation*}
where we have used
\begin{equation*}
f\mu\mu'=f\frac{(\mu^2)'}{2}=\frac12\left((f\mu^2)'-f'\mu^2\right).
\end{equation*}
Following almost the same argument, we can derive \eqref{2.11} by replacing \eqref{2.12}-\eqref{2.14} and \eqref{2.16} with
\begin{equation*}
\tilde \mu'(t)=\int_{\{u=t\}}\frac{1}{|\nabla u|}d\mathcal{H}^1,~\tilde F'(t)=f(t)\tilde \mu'(t),~\tilde G'(t)=g(t)\tilde \mu'(t),
\end{equation*}
and
\begin{equation*}
\tilde F(t)=\tilde G(t)-\int_{\{u=t\}}|\nabla u|d\mathcal{H}^1.
\end{equation*}
\end{proof}

\section{An integral inequality}
In this section, we shall derive an important integral inequality, which plays a key role in the proof of our uniqueness result.
\begin{proposition}
\label{pr3.1}
Assume \eqref{1.fga} holds and the set $\mathcal{C}(f,g)$ is finite. The any solution $u$ of \eqref{2.3} satisfies
\begin{equation}
\label{3.1}
\frac{A_g}{2}\int_{-\infty}^\infty f'(t)\mu(t)(|\Omega|-\mu(t))dt\geq\int_{-\infty}^\infty f(t)I_{\Omega}^2(\mu(t))dt.
\end{equation}
\end{proposition}

\begin{proof}
For any solution $u$ of \eqref{2.3}, we set
\begin{equation*}
t_0:=\min_{\Omega}u,~t_1:=\max_{\Omega}u,~F_{\Omega}=\int_{\Omega}f(u)dx.
\end{equation*}
The result of Proposition \ref{pr3.1} will follow by integrating the differential inequalities \eqref{2.10}-\eqref{2.11} on the interval $(t_0,t_1)$ as shown below.

Let us consider the functions
\begin{equation}
\label{3.psi}
\Psi:=A_g f\mu^2-F^2\quad\mathrm{and}\quad \tilde\Psi:=A_g f\tilde \mu^2-\tilde F^2.
\end{equation}
By Lemma \ref{le2.1}, the functions $\Psi$ and $\tilde\Psi$ are continuous on any interval $[a,b]\subset\mathbb{R}\setminus\mathcal{C}(f,g)$. At the point $t\in\mathcal{C}(f,g)$, these functions may be discontinuous, and hence we have to treat it separately.

For each $a\in \mathcal{C}(f,g)$, we set $\Gamma_a:=\{u=a\}$, and claim
\begin{equation}
\label{3.2}
\begin{aligned}
&\Psi(a^+)+\tilde\Psi(a^+)-\Psi(a^-)-\tilde\Psi(a^-)\\
&=2f(a)|\Gamma_a|\left[(F(a)-\mu(a)A_g)+(\tilde \mu(a)A_g-\tilde F(a))\right]\\
&\geq0,
\end{aligned}
\end{equation}
where $\Psi(a^\pm):=\lim_{\varepsilon\to0}\Psi(a\pm\varepsilon)$, $\tilde\Psi(a^\pm):=\lim_{\varepsilon\to0}\tilde\Psi(a\pm\varepsilon).$

Indeed, according to the definition of $\Psi$ and $\tilde\Psi$, we have
\begin{equation}
\label{3.3}
\begin{aligned}
\Psi(a^+)-\Psi(a^-)=~&A_gf(a)[\mu(a)^2-(\mu(a)+|\Gamma_a|)^2]\\
&+(F(a)+f(a)|\Gamma_a|)^2-F(a)^2,
\end{aligned}
\end{equation}
and
\begin{equation}
\label{3.4}
\begin{aligned}
\tilde\Psi(a^+)-\tilde\Psi(a^-)=~&A_gf(a)[(\tilde \mu(a)+|\Gamma_a|)^2-\tilde \mu(a)^2]\\
&+\tilde F(a)^2-(\tilde F(a)+f(a)|\Gamma_a|)^2.
\end{aligned}
\end{equation}
Adding \eqref{3.3} and \eqref{3.4} together, we have
\begin{equation}
\label{3.5}
\begin{aligned}
&\Psi(a^+)+\tilde\Psi(a^+)-\Psi(a^-)-\tilde\Psi(a^-)\\
&~=2f(a)|\Gamma_a|\left[(F(a)-\mu(a)A_g)+(\tilde \mu(a)A_g-\tilde F(a))\right].
\end{aligned}
\end{equation}
It is easy to see that $A_g=A_f$ from the equation \eqref{2.3} and the Neumann boundary condition. Using Lemma \ref{le2.2}, we get
\begin{align*}
\frac{F(a)}{\mu(a)}\geq\frac{\int_{\Omega}f(u)dx}{|\Omega|}=A_f=A_g
\quad\mathrm{and}\quad
\frac{\tilde F(a)}{\tilde \mu(a)}\leq\frac{\int_{\Omega}f(u)dx}{|\Omega|}=A_f=A_g,
\end{align*}
which imply $F(a)-\mu(a)A_g\geq0$ and $\tilde \mu(a)A_g-\tilde F(a)\geq0$. As a consequence, the right hand side of \eqref{3.5} is non-negative. Thus the claim \eqref{3.2} is proved.

By \eqref{2.10} and \eqref{2.11}, we have both $\Psi$ and $\tilde\Psi$ are monotone increasing on the intervals $[t_0,t_1]\setminus\bigcup_{a\in \mathcal{C}(f,g)}(a-\varepsilon,a+\varepsilon)$ for $\varepsilon>0$. Then we get
\begin{equation}
\label{3.6}
\int_{t_0}^{t_1}\Psi'dt\leq \sum_{a\in\mathcal{C}(f,g)}(\Psi(a^{-})-\Psi(a^+))+F_{\Omega}^2-A_gf(t_0)|\Omega|^2.
\end{equation}
Similarly, we derive that
\begin{equation}
\label{3.7}
\int_{t_0}^{t_1}\tilde\Psi'dt\leq \sum_{a\in\mathcal{C}(f,g)}(\tilde\Psi(a^{-})-\tilde\Psi(a^+))+A_gf(t_1)|\Omega|^2-F_{\Omega}^2.
\end{equation}
By Proposition \ref{pr2.1}, \eqref{3.6}-\eqref{3.7} and \eqref{3.2}, we get
\begin{equation}
\label{3.8}
\begin{aligned}
A_g|\Omega|^2(f(t_1)-f(t_0))\geq~&\int_{t_0}^{t_1}\left(A_gf'(t)\mu^2(t)+2f(t)I_{\Omega}^2(\mu(t))\right)dt\\
&+\int_{t_0}^{t_1}\left(A_gf'(t)\tilde \mu^2(t)+2f(t)I_{\Omega}^2(\tilde \mu(t))\right)dt\\
=~&\int_{t_0}^{t_1}A_gf'(t)\left(\mu(t)^2+(|\Omega|-\mu(t))^2\right)dt\\
&+4\int_{t_0}^{t_1}f(t)I_{\Omega}^2(\mu(t))dt\\
=~&A_g|\Omega|^2\int_{t_0}^{t_1}f'(t)dt+4\int_{t_0}^{t_1}f(t)I_{\Omega}^2(\mu(t))dt\\
&+2A_g\int_{t_0}^{t_1}f'(t)\mu(t)(\mu(t)-|\Omega|)dt
\end{aligned}
\end{equation}
where we used $\tilde \mu=|\Omega|-\mu$ and $I_\Omega(\mu)=I_\Omega(|\Omega|-\mu)$. Since $f$ is differentiable, inequality \eqref{3.8} yields:
\begin{equation*}
\label{3.9}
0\geq A_g\int_{t_0}^{t_1}f'(t)\mu(t)(\mu(t)-|\Omega|)dt+2\int_{t_0}^{t_1}f(t)I_\Omega^2(\mu(t))dt,
\end{equation*}
or equivalently
\begin{equation*}
\label{3.10}
A_g\int_{t_0}^{t_1}f'(t)\mu(t)(|\Omega|-\mu(t))dt\geq 2\int_{t_0}^{t_1}f(t)I_\Omega^2(\mu(t))dt,
\end{equation*}
which proves \eqref{3.1}.
\end{proof}

\noindent {\bf Remark 3.1:} Instead of using the inequality $\mathcal{H}^1(\{u=t\}\cap\Omega)\geq I_{\Omega}(\mu(t))$ from \eqref{2.18}, we can keep the term $\mathcal{H}^1(\{u=t\})$ and repeat the arguments of deriving \eqref{3.1} to get
\begin{equation}
\label{3.11}
\frac{A_g}{2}\int_{-\infty}^\infty f'(t)\mu(t)(|\Omega|-\mu(t))dt\geq\int_{-\infty}^\infty f(t)[\mathcal{H}^1(\{u=t\}\cap\Omega)]^2dt.
\end{equation}

In section 4, we will get the uniqueness result asserted in Theorem \ref{th1.1} from the above inequality \eqref{3.11}.
\medskip

\noindent {\bf Remark 3.2:} When $\Omega$ is replaced by the mainfold $M$, we can also derive a counterpart result of Proposition \ref{pr3.1}. Then a similar result of \cite[Theorem 1.1]{l} and some related conclusions can be obtained by the same arguments (see \cite[section 3]{l} for more details).

\section{The uniqueness of the Neumann problem \eqref{1.6}}
In this section, we shall apply the inequalities established in the preceding section to prove the main results of the Neumann problem \eqref{1.6}.

With
\begin{equation}
\label{4.1}
f(u)=\lambda\frac{e^u}{\int_{\Omega}e^udx}\quad\mathrm{and}\quad g(u)=\beta u,
\end{equation}
the problem \eqref{2.3} is turned to be
\begin{equation}
\label{4.2}
\begin{cases}
\Delta u-\beta u+\lambda\frac{e^u}{\int_{\Omega}e^udx}=0~&\mathrm{in}~\Omega,\\
u>0~&\mathrm{in}~\Omega,\\
\partial_{\nu}u=0~&\mathrm{on}~\Omega.
\end{cases}
\end{equation}
It is easy to check that $f,g$ verify the condition \eqref{1.fga}. In the next result we shall prove
$$|\mathcal{C}(f,g)|=\left|\mathcal{C}(\lambda\frac{e^u}{\int_{\Omega}e^udx},\beta u)\right|<+\infty.$$

\begin{lemma}
Let $f,g$ be defined in \eqref{4.1}. Then
$$1\leq|\mathcal{C}(f,g)|\leq 2.$$
\end{lemma}

\begin{proof}
We divide our proof into two steps.
\medskip

\noindent Step 1. $|\mathcal{C}(f,g)|\geq1$. Using the boundary condition, we have
\begin{equation}
\label{4.3}
\int_{\Omega}f(u)dx=\beta\int_{\Omega}udx.
\end{equation}
Then we claim that $f(u(x))=\beta u(x)$ must happen for some $x\in\Omega.$ Otherwise, we have either $f(u(x))>\beta u(x)$ or $f(u(x))<\beta u(x)$ for all $x\in\Omega$, which leads to $\int_{\Omega}f(u)dx>\beta\int_{\Omega}udx$ or $\int_{\Omega}f(u)dx<\beta\int_{\Omega}udx$, it contradicts to \eqref{4.3}. Thus $\{u\mid f(u)=\beta u\}\neq\emptyset.$
\medskip

\noindent Step 2. $|\mathcal{C}(f,g)|\leq2$. Considering the function $h(x)=f(x)-\beta x$, which is convex by noticing that $f$ is a convex function. As a consequence, the function $h$ possesses at most two roots. Hence, we finish the proof.
\end{proof}

Next we prove the following theorem, which is equivalent to Theorem \ref{th1.1}.

\begin{theorem}
\label{th4.1}
Let $u$ be a non-constant solution of problem \eqref{4.2} with $\mathrm{osc}_{\partial\Omega}(u)=0$. Then the following inequality holds:
\begin{equation*}
\label{4.4}
\beta\int_{\Omega}udx>8\pi.
\end{equation*}
\end{theorem}

\begin{proof}
In Remark 3.1, we have pointed out that the following inequality,
\begin{equation}
\label{4.5}
\frac{A_g}{2}\int_{-\infty}^{\infty}f'(t)\mu(t)(|\Omega|-\mu(t))dt\geq\int_{-\infty}^{\infty}f(t)[\mathcal{H}^1(\{u=t\}\cap\Omega)]^2dt.
\end{equation}
By the setting of $g$, we have $A_g=\beta\bar u$. Setting $u_0:=u|_{\partial\Omega}$, we consider the regular values of $u$:
\begin{equation*}
\mathrm{Reg}(u):=\{t\in\mathbb{R}:\nabla u(x)\neq0,~\forall x\in u^{-1}(t)\}.
\end{equation*}
From Sard's theorem we have $\mathbb{R}\setminus\mathrm{Reg}(u)$ has zero Lebesgue measure, which along with the implicit function theorem implies that
\begin{equation*}
\{u=t\}\Subset\Omega,~\{u=t\}~\mbox{is a 1-submanifold of class}~C^1~\mathrm{for}~ t\in \mbox{Reg}(u)\setminus\{u_0\}.
\end{equation*}
Next we claim that $\{u<t\}$ or $\{u>t\}$ must be contained in a domain enclosed by some connected branch of the set $\{u=t\}$. Indeed, for any $t\neq u_0,$ without loss of generality we may assume $t<u_0$. Then it is evident to see that $\{u<t\}\cap\partial\Omega=\emptyset$ and the claim holds. As a consequence
$$\mathcal{H}^1(\{u=t\}\cap\Omega)=\mathcal{H}^1(\{u=t\}).$$
Next, by using the isoperimetric inequality we have
\begin{equation}
\label{4.6}
[\mathcal{H}^1(\{u=t\})]^2\geq4\pi\min\{\mu(t),|\Omega-\mu(t)|\},~\forall t\in\mathrm{Reg}(u)\setminus\{u_0\}.
\end{equation}
From \eqref{4.6} we get
\begin{equation}
\label{4.7}
\frac{f(t)[\mathcal{H}^1(\{u=t\})]^2}{f'(t)\mu(t)(|\Omega|-\mu(t))}\geq\frac{4\pi}{\max\{\mu(t),|\Omega|-\mu(t)\}}>\frac{4\pi}{|\Omega|},
~t\in(t_0,t_1).
\end{equation}
Combining with the fact $\mathbb{R}\setminus\mathrm{Reg}(u)$ has zero Lebesgue measure, \eqref{4.5} and \eqref{4.7}, we have
\begin{equation*}
\label{4.8}
\beta\bar u>\frac{8\pi}{|\Omega|},
\end{equation*}
which yields the result.
\end{proof}

\textcolor{black}{Before proving Theorem \ref{th.rotation}, we shall give the following result on the general solutions of \eqref{1.6} on the unit disc.}

\begin{proposition}
\label{pr4.1new}
Let $\Omega$ be the unit disc. If $u$ is a non-constant solution solving the problem \eqref{1.6}, then
\begin{equation*}
\lambda>\frac{32}{\pi}.
\end{equation*}
\end{proposition}

\noindent{\em Proof of Proposition \ref{pr4.1new}.} \textcolor{black}{When $\Omega$ is a disc, it is proved in \cite[18.1.3]{bz} that
\begin{equation*}
\frac{I_{\Omega}^2(s)}{s(|\Omega|-s)}>\frac{16}{\pi|\Omega|},~\forall s\in(0,\frac{|\Omega|}{2}).
\end{equation*}
Using the condition $f\geq f'>0$, we have
\begin{equation*}
\frac{I_{\Omega}^2(s)f(s)}{s(|\Omega|-s)f'(s)}>\frac{16}{\pi|\Omega|},
\end{equation*}
which along with Proposition \ref{pr3.1} implies
\begin{equation*}
\beta\bar u>\frac{32}{\pi|\Omega|}.
\end{equation*}
On the other hand, we have $\lambda=\beta\int_{\Omega}u$. Thus, we get $\lambda>\frac{32}{\pi}$ and it proves the conclusion.}
\hfill $\square$
\medskip

In the following, we shall consider the case that $u$ is invariant under a rotation $R_{\theta}$. To study the class of functions which are invariant by a rotation $R_{\frac{2\pi}{m}}$, we recall a definition in \cite{hl,l}:
\medskip

\noindent {\bf Definition 4.1} (isoperimetric profile)  Given a group $G$ of isometries of $\Omega$, we consider the class of open subsets:
\begin{equation*}
\mathcal{O}_{\Omega}^G:=\{\omega\subset\mathcal{O}_{\Omega}:g(\omega)=\omega,\forall g\in G\}.
\end{equation*}
The ``G-isoperimetric profile" of $\Omega$ is defined as
\begin{equation*}
I_{\Omega}^G(s):=\inf\left\{\mathcal{H}^1(\partial\omega\cap\Omega):\omega\in\mathcal{O}_{\Omega}^G,\mathcal{H}^2(\omega)=s\right\},~\forall s\in(0,|\Omega|].
\end{equation*}
We set $I_{\Omega}^G(0)=0.$
\medskip

In the setting of ``G-isoperimetric profile", we can generalize the Proposition \ref{pr3.1} to the following result:
\begin{proposition}
\label{pr4.2}
Let $\Omega\Subset\mathbb{R}^2$ be a piecewise $C^1$-domain. If $G$ is a group of isometry of $\Omega$ and $u\in C^2(\overline\Omega)\cap H^G$ is a non-constant solution of \eqref{4.2}, then
\begin{equation*}
\label{4.9}
\frac{\beta\bar u}{2}\int_{-\infty}^{\infty}f'(t)\mu(t)(|\Omega|-\mu(t))dt\geq\int_{-\infty}^{\infty}f(t)[I_{\Omega}^G(\mu(t))]^2dt.
\end{equation*}
\end{proposition}

\begin{proof}
We can follow the proof of Proposition \ref{pr3.1} step by step to prove Proposition \ref{pr4.2}, just by noticing that for almost every $t\in\mathbb{R}$ the level sets $\{u>t\}$ and $\{u<t\}$ belong to $\mathcal{O}_{\Omega}^G$.
\end{proof}

To continue our discussion, we need the following two lemmas. For the proof we refer the readers to \cite[Lemma 4.6-4.7]{hl}.
\begin{lemma}
\label{le4.2}
Let $B$ be a disc of $\mathbb{R}^2$ and set $G=\langle R_{2\pi/m}\rangle,~m\geq2$.

(a) If $\omega\in\mathcal{O}_B^G$ is such that
\begin{equation*}
\label{4.10}
(i)~\mathcal{H}^1(\overline\omega\cap\partial B)=0~\mathrm{or}~(ii)~\mathcal{H}^1(\overline{B\setminus\omega}\cap\partial B)=0,
\end{equation*}
then $\mathcal{H}^1(\partial\omega\cap B)\geq\min\{(4\pi\omega)^{\frac12},(4\pi[|B|-|\omega|])^{\frac12}\}.$

(b) If $\omega=\bigcup_{i=0}^{m-1}R_{\frac{2\pi}{m}}^i(\omega_0)$ (disjoint union) with $w_0$ satisfying
\begin{equation}
\label{4.11}
\omega_0\in\mathcal{O}_B,~\mathcal{H}^1(\overline{\omega_0}\cap\partial B)>0~\mathrm{and}~
\mathcal{H}^1(\overline{B\setminus\omega_0}\cap\partial B)>0,
\end{equation}
then
\begin{equation}
\label{4.12}
\mathcal{H}^1(\partial\omega\cap B)\geq nI_B\left(\frac{|\omega|}{m}\right).
\end{equation}
If $B\setminus\omega=\bigcup_{i=0}^{m-1}R_{\frac{2\pi}{m}}^i(\omega_0)$ with $\omega_0$ satisfying \eqref{4.11}, then \eqref{4.12} still holds.
\end{lemma}

\begin{lemma}
\label{le4.3}
For a disc $B\subset\mathbb{R}^2$ and $G=\langle R_{2\pi/m}\rangle$ it holds
\begin{equation*}
\label{4.13}
\frac{[I_B^G(s)]^2}{s(|B|-s)}>\min\left\{\frac{4\pi}{|B|},\frac{16m}{\pi |B|}\right\}.
\end{equation*}
\end{lemma}
\medskip

With the above two lemmas, we can prove Theorem \ref{th.rotation}.
\medskip

\noindent {\em Proof of Theorem \ref{th.rotation}.} For the solution $u\in H^G$ with $G=\langle R_{2\pi/m}\rangle$, each of the level set $\Omega_t:=\{u>t\}$ is invariant under the action of the group. Then we deduce from Proposition \ref{pr4.2} that
\begin{equation*}
\label{4.14}
\frac{\beta\int_{\Omega}udx}{2|\Omega|}\int_{-\infty}^{\infty}f'(t)\mu(t)(|\Omega|-\mu(t))dt
\geq\int_{-\infty}^{\infty}f(t)[I_{\Omega}^G(\mu(t))]^2dt,
\end{equation*}
with $f(u)=\lambda\frac{e^u}{\int_{\Omega}e^udx}$ and $\Omega=B$. Using Lemma \ref{le4.3}, we get
\begin{equation*}
\label{4.15}
\frac{\lambda }{2|\Omega|}>
\begin{cases}
\frac{32}{\pi|\Omega|},~&\mbox{if}~m=2,\\
\frac{4\pi}{|\Omega|},~&\mbox{if}~m\geq3.
\end{cases}
\end{equation*}
Then
\begin{equation*}
\label{4.16}
\lambda>
\begin{cases}
\frac{64}{\pi},~&\mbox{if}~m=2,\\
8\pi,~&\mbox{if}~m\geq3,
\end{cases}
\end{equation*}
which finishes the whole proof. \hfill $\square$
\medskip

Next we shall apply Theorem \ref{th.rotation} to derive the optimal inequalities for the functional $J_{\lambda}(u)$:
\begin{align*}
J_{\lambda}(u)=\frac12\int_{\Omega}|\nabla u|^2dx+\frac\beta2\int_{\Omega}|u|^2dx-\lambda\log\left(\int_{\Omega}e^udx\right).
\end{align*}

\begin{proposition}
\label{pr4.3}
Let $J_{\lambda}^m(\cdot)$ be the restriction of $J_{\lambda}(\cdot)$ to the space $H^G$ with $G=\langle R_{2\pi/m}\rangle,~m\geq2$. Then the following hold:
\begin{enumerate}
  \item[(a)] The functional $J_{\lambda}^m(\cdot)$ is bounded from below whenever $\lambda\leq 8\pi.$
  \item[(b)] If $m\geq3$ and $\lambda\leq 8\pi$, the functional $J_{\lambda}^m(\cdot)$ admits a unique global minimizer given by $u\equiv\frac{\lambda}{\beta|\Omega|}$.
  \item[(c)] For $m=2$, the functional $J_{\lambda}^m(\cdot)$ admits a global minimizer for each $\lambda\leq8\pi$. Furthermore, the this global minimizer is unique and given by $\frac{\lambda}{\beta|\Omega|}$ whenever $\lambda\leq\frac{64}{|\Omega|}.$
\end{enumerate}
\end{proposition}

\begin{proof}
Let us first prove that when $\lambda\leq 8\pi$ we may find a constant $C>0$ depending only on $\lambda$ and $|\Omega|$ such that
\begin{align*}
J_{\lambda}(u)\geq-C,~\forall u\in H^{\langle R_{2\pi/m}\rangle}.
\end{align*}
We need the following inequality. For a bounded domain $\Omega$ of $\mathbb{R}^2$ whose boundary is $C^2$-piecewise with finite number of vertexes, denote by $\theta_{\Omega}$ the minimum interior angle among all the vertexes. Then we have the following Moser-Trudinger inequality (cf. \cite{cy,c}):
\begin{equation*}
\label{4.17}
\int_{\Omega}e^{2\theta_{\Omega}\left(\frac{u-\bar u}{\|\nabla u\|_2}\right)^2}dx\leq C_0,~\forall u\in H^1(\Omega).
\end{equation*}
As a consequence, we have
\begin{equation*}
\begin{aligned}
\log\int_{\Omega}e^{u}dx&=\frac{1}{|\Omega|}\int_{\Omega}udx+\log\int_{\Omega}e^{u-\bar u}dx\\
&\leq\frac{1}{|\Omega|}\int_{\Omega}udx+\log\int_{\Omega}e^{\frac{1}{8\theta_{\Omega}}\|\nabla u\|_2^2
+2\theta_{\Omega}\frac{(u-\bar u)^2}{\|\nabla u\|_2^2}}dx\\
&\leq\frac{\beta}{8\theta_{\Omega}}\int_{\Omega}|u|^2dx+\frac{1}{8\theta_{\Omega}}\int_{\Omega}|\nabla u|^2dx+\frac{2\theta_{\Omega}}{\beta|\Omega|}+\log C_0,
\end{aligned}
\end{equation*}
and it implies that
\begin{equation}
\label{4.18}
\frac12\int_{\Omega}|\nabla u|^2dx+\frac\beta2\int_{\Omega}|u|^2dx-4\theta_{\Omega}\log\left(\int_{\Omega}e^udx\right)\geq -C_1,~u\in H^1(\Omega).
\end{equation}

Now fixing $P\neq 0$, we consider a fundamental domain
\begin{equation*}
\label{4.19}
\Pi(P,\theta):=\left\{x\in B\setminus\{0\}:0<\arccos\left(\frac{x\cdot OP}{\|x\|\|OP\|}\right)<\frac{\theta}{2}\right\},
\end{equation*}
and split $\Pi(P,\theta)$ as follows:
\begin{equation*}
\label{4.20}
\Pi_+(P,\theta):=\left\{x\in\Sigma:x\cdot R_{\frac{\pi}{2}}(OP)>0\right\},
\end{equation*}
\begin{equation*}
\label{4.21}
\Pi_-(P,\theta):=\left\{x\in\Sigma:x\cdot R_{\frac{\pi}{2}}(OP)<0\right\}.
\end{equation*}
In $\Pi$, for any $u\in H^{\langle R_{2\pi/m}\rangle}$, we have
\begin{equation}
\label{4.22}
\begin{aligned}
J_{\lambda}(u):&=\frac12\int_{B}|\nabla u|^2dx+\frac\beta2\int_{B}|u|^2dx-\lambda\log\left(\int_{B}e^udx\right)\\
&=m\left(\frac12\int_{\Pi}|\nabla u|^2dx+\frac\beta2\int_{\Pi}|u|^2dx-\frac{\lambda}{m}\log\left(\int_{\Pi}e^udx\right)\right).
\end{aligned}
\end{equation}
For the domain $\Pi$, we apply the inequality \eqref{4.18} with $\theta_{\Pi}$  given by
\begin{equation*}
\label{4.23}
\theta_{\Pi}=\begin{cases}\frac{\pi}{2},~&\mbox{if}~m=2,3,\\\frac{2\pi}{m},~&\mbox{if}~m\geq4.\end{cases}
\end{equation*}
Then we see that \eqref{4.22} is uniformly bounded from below if
\begin{equation*}
\label{4.24}
\frac{\lambda}{m}\leq 4\theta_{\Pi}=\begin{cases}2\pi,~&\mbox{if}~m=2,3,\\\frac{8\pi}{m},~&\mbox{if}~m\geq4,\end{cases}
\end{equation*}
which is equivalent to
\begin{equation*}
\label{4.25}
\lambda\leq\begin{cases}2m\pi,~&\mbox{if}~m=2,3,\\ 8\pi,~&\mbox{if}~m\geq4.\end{cases}
\end{equation*}
This proves that \eqref{4.22} is uniformly bounded from below provided $\lambda\leq 8\pi$ and $m\geq 4$.

For $m=2,3$, using the Steiner symmetrization, we can assume that the function $u$ is radial. As a consequence, the fundamental domain $\Pi$ is replaced by $\Pi_+$ and
\begin{equation}
\label{4.26}
\begin{aligned}
J_{\lambda}(u):&=\frac12\int_{B}|\nabla u|^2dx+\frac\beta2\int_{B}|u|^2dx-\lambda\log\left(\int_{B}e^udx\right)\\
&=2m\left(\frac12\int_{\Pi_+}|\nabla u|^2dx+\frac\beta2\int_{\Pi_+}|u|^2dx-\frac{\lambda}{2m}\log\left(\int_{\Pi}e^udx\right)\right).
\end{aligned}
\end{equation}
Notice that in $\Pi_+$, the constant $\theta_{\Pi_+}=\frac{\pi}{m}.$ Using \eqref{4.18}, we see \eqref{4.26} is bounded from below if
\begin{equation*}
\label{4.27}
\frac{\lambda}{2m}\leq4\frac{\pi}{m},
\end{equation*}
which implies $\lambda\leq 8\pi$. Thus we have deduced that the functional $J_{\lambda}(u)$ restricted to the space $H^G$ is bounded from below when $\lambda\leq8\pi.$ From the standard variational argument we can find $u_{\min}\in H^{\langle R_{2\pi/m}\rangle}$ such that
\begin{equation*}
\label{4.28}
J_{\lambda}(u)\geq J_{\lambda}(u_{\min}),~\forall u\in H^{\langle R_{2\pi/m}\rangle}.
\end{equation*}
Then the remaining assertion of Proposition \ref{pr4.3} is a direct consequence of Theorem \ref{th.rotation}.
\end{proof}

Now we ready to prove Theorem \ref{th1.4}.
\medskip


\noindent {\em Proof of Theorem \ref{th1.4}.} Applying Proposition \ref{pr4.3} and the arguments in \cite[Theorem 1.1]{nsy}, we can conclude that any solution $(v,u)\in H^G\times\mathring{H}^G$ of \eqref{1.1} is globally defined whenever $\lambda<\Lambda_m$ and show that
\begin{equation*}
\lim_{t\to\infty}\sup\left(\|v(x,t)\|_{L^\infty(\Omega)}+\|u(x,t)\|_{L^\infty(\Omega)}\right)<\infty.
\end{equation*}
From the result \cite[Theorem 1.1]{fp}, we get that the classical solutions converge in $C^1(\overline B)$ as $t\to\infty$ to a stationary solution. Particularly, this convergence holds true for a subsequence $(t_k)_{k\in\mathbb{N}}$ in the sense that
$$v(t_k)\to\frac{e^{u_\infty}}{\int_{\Omega}e^{u_\infty}dx}~\mathrm{in}~L^2(B)~\mathrm{as}~t_k\to\infty,$$
and
$$u(t_k)\to u_{\infty}~\mathrm{in}~H^1(B)~\mathrm{as}~t_k\to\infty,$$
where $u_{\infty}$ is a solution of \eqref{1.6}. It is known by Theorem \ref{th.rotation} that $u=\frac{\lambda}{\beta\pi}$ is the only solution to the problem \eqref{1.6} provided $\lambda<\Lambda_m$. Thus, a solution $(v,u)\in H^G\times\mathring{H}^G$ of system \eqref{1.1} must converge to the constant solution $(\frac{\lambda}{\pi},\frac{\lambda}{\beta\pi})$ as $t\to\infty$ when $\lambda<\Lambda_m.$ Finally, if $\lambda<8\pi$, the convergence to the constant equilibrium of any radial solution of \eqref{1.1} results from Remark 1.1. \hfill $\square$

\section{Uniqueness of the Dirichlet problem \eqref{1.dirichlet}}
In this section, we shall provide a complete proof of Theorem \ref{th.dirichlet}. Indeed, we may consider a more general problem:

\begin{equation}
\label{5.1}
\begin{cases}
\Delta u-g(u)+\lambda\frac{e^u}{\int_{\Omega}e^udx}=0~&\mathrm{in}~\Omega,\\
u>0~&\mathrm{in}~\Omega,\\
u=0~&\mathrm{on}~\partial\Omega,
\end{cases}
\end{equation}
where $g$ satisfies that $g(x),g'(x)>0$ for $x>0.$

Concerning the problem \eqref{5.1}, we obtain the following conclusion,
\begin{theorem}
\label{th5.1}
Let $\Omega$ be an open bounded and simply connected set. Assuming that $u_i\in C^2(\Omega)\cap C(\overline{\Omega}) (i=1,2)$ solve the equation \eqref{5.1} and  $u_1\not\equiv u_2$, then $\lambda>8\pi.$
\end{theorem}

\textcolor{black}{We prove the above result by following the same spirit as treated for the Neumann problem \eqref{1.6}. Precisely, we focus on the difference between the integration of $g(u_2)$ and $g(u_1)$ with respect to the level set of $u_2-u_1,$ instead of the pointwise comparison between $g(u_2)$ and $g(u_1).$} To begin with the argument, let us recall the classical Bol's isoperimetric inequality, see \cite{ba,bl,b} and \cite{bz} for a detailed history of the Bol's inequality.
\medskip

\noindent {\bf Theorem 5.A.} {\em Let $\Omega\subset\mathbb{R}^2$ be a simply-connected and assume $w\in C^2(\Omega)\cap C(\overline {\Omega})$ satisfies
\begin{equation*}
\label{5.2}
\Delta w+e^w\geq0,~\int_{\Omega}e^wdx\leq8\pi.
\end{equation*}
Then for every $\omega\Subset\Omega$ of class $C^1$ the following inequality holds
\begin{equation*}
\left(\int_{\partial\omega}e^{\frac{w}{2}}dx\right)^2\geq\frac12\left(\int_{\omega}e^wdx\right)
\left(8\pi-\int_{\omega}e^wdx\right).
\end{equation*}
Moreover the above inequality is strict if $\Delta w+e^w>0$ somewhere in $\omega$ or $\omega$ is not simply connected.}
\medskip

For $\theta>0$, the function defined by
\begin{equation}
\label{5.3}
U_{\theta}(x)=-2\log\left(1+\frac{\theta^2|x|^2}{8}\right)+2\log\theta
\end{equation}
satisfies the following property
$$\Delta U_{\theta}+e^{U_{\theta}}=0,~\mathrm{and}~
\left(\int_{\partial B_r}e^{\frac{U_{\theta}}{2}}dx\right)^2=\frac12\left(\int_{B_r}e^{U_{\theta}}dx\right)
\left(8\pi-\int_{B_r}e^{U_{\theta}}dx\right),$$
for all $r>0$ and $\theta>0,$ where $B_r$ denotes the ball of radius $r$ centered at the origin in $\mathbb{R}^2$.

Next we shall recall some facts about the rearrangement with the measures. Such discussions have been detailed in \cite{ba,bl,bl1,ccl,gm0,s}, we shall sketch the process here only. For any function $\phi\in C^2(\overline{\Omega})$ which is constant on $\partial\Omega$ can be equimeasurably rearranged with respect to the measures $e^{u}dx$ and $e^{U_{\theta}}dx$, where $u$ is a function satisfying that $\Delta u+e^{u}\geq0$ and $U_{\theta}$ is defined in \eqref{5.3}. For any $t>\min_{x\in\overline{\Omega}}\phi$, we define $\Omega_t^*$ be a ball centered at the origin such that
\begin{equation*}
\int_{\Omega_t^*}e^{U_{\theta}}dx=\int_{\{\phi>t\}}e^udx.
\end{equation*}
Then we define $\phi^*:\Omega^*\to\mathbb{R}$ by
$$\phi^*(x):=\sup\{t\in\mathbb{R}:x\in\Omega_t^*\},$$
which gives an equimeasurable rearrangement of $\phi$ with respect to the measures $e^{u}dx$ and $e^{U_{\theta}}dx$:
$$\int_{\{\phi^*>t\}}e^{U_{\theta}}dx=\int_{\{\phi>t\}}e^{u}dx,~\forall t>\min_{x\in\overline{\Omega}}\phi.$$
For functions $\phi$ and $\phi^*$, by using Proposition 5.A, the following conclusions hold.
\medskip

\noindent{\bf Proposition 5.B.} \cite{gm0} {\em Let $u\in C^{0,1}(\overline{\Omega})$ satisfy
\begin{equation*}
\Delta u+e^u\geq0~\mbox{in}~\Omega,
\end{equation*}
and let $U_{\theta}$ be given in \eqref{5.3}. Suppose $\phi\in C^2(\overline{\Omega})$ and $\phi\equiv C$ on $\partial\Omega$. Let $\phi^*$ be the equimeasurable symmetric rearrangement of $\phi$ with respect to the measures $e^{u}dx$ and $e^{U_{\theta}}dx$, then it holds for all $t>\min_{x\in\overline{\Omega}}\phi(x)$ that
\begin{equation*}
\int_{\{\phi^*=t\}}|\nabla\phi^*|d\sigma\leq \int_{\{\phi=t\}}|\nabla\phi|d\sigma.
\end{equation*}}

The following two lemmas are the consequences of the Bols's inequality and reversed Bol's inequality in the radial setting respectively.

\medskip
\noindent{\bf Lemma 5.C.} \cite{gm0} {\em Assume that $\psi\in C^{0,1}(\overline{B_R})$ is a strictly decreasing, radial, Lipschitz function, and satisfies
\begin{equation}
\label{5.4}
\int_{\partial B_r}|\nabla\psi|d\sigma\leq\int_{B_r}e^{\psi}dx
\end{equation}
a.e. $r\in(0,R)$ and $\psi=U_{\theta_1}=U_{\theta_2}$ for some $\theta_2>\theta_1$ on $\partial B_R$. Then there holds
\begin{equation}
\label{5.5}
\mbox{either}\quad \int_{B_R}e^{\psi}dx\leq\int_{B_R}e^{U_{\theta_1}}dx\quad\mbox{or}\quad\int_{B_R}e^{\psi}dx\geq \int_{B_R}e^{U_{\theta_2}}dx.
\end{equation}
Moreover if the inequality in \eqref{5.4} is strict in a set with positive measure in $(0,R)$, then the inequalities in \eqref{5.5} are also strict.}
\medskip

\noindent{\bf Lemma 5.D.} \cite{gm} {\em Assume that $\psi\in C^{0,1}(\mathbb{R}^2\setminus B_R)$ is a strictly decreasing, radial, Lipschitz function, and satisfies
\begin{equation}
\label{5.6}
\int_{\partial B_r}|\nabla\psi|d\sigma\leq8\pi-\int_{\mathbb{R}^2\setminus B_r}e^{\psi}dx
\end{equation}
a.e. $r\in(R,+\infty)$ and $\psi=U_{\theta_1}=U_{\theta_2}$ for some $\theta_2>\theta_1$ on $\partial B_R$. Then there holds
\begin{equation}
\label{5.7}
\int_{\mathbb{R}^2\setminus B_R}e^{U_{\theta_2}}dx\leq\int_{\mathbb{R}^2\setminus B_R}e^{\psi}dx\leq \int_{\mathbb{R}^2\setminus B_R}e^{U_{\theta_1}}dx.
\end{equation}
Moreover if the inequality in \eqref{5.6} is strict in a set with positive measure in $(R,+\infty)$, then the inequalities in \eqref{5.7} are also strict.}
\medskip

We shall also need the following lemma.

\medskip
\noindent{\bf Lemma 5.E.} \cite{gm} {\em Assume that $\psi\in C^{0,1}(\overline{B_R})$ is a strictly decreasing and radial function satisfying \eqref{5.4} for a.e. $r\in(0,R)$. If
\begin{equation*}
\int_{B_R}e^{\psi}dx=\int_{B_R}e^{U_{\theta}}dx<8\pi,
\end{equation*}
then $U_{\theta}(R)\leq\psi(R).$}
\medskip

Now we are ready to prove Theorem \ref{th5.1}.
\medskip

\noindent {\em Proof of Theorem \ref{th5.1}.} Without loss of generality, we may assume $\int_{\Omega}g(u_2)dx\geq\int_{\Omega}g(u_1)dx.$ Let
\begin{equation*}
\label{5.8}
v_i=u_i+\log\lambda-\log\int_{\Omega}e^{u_i}dx,~i=1,2.
\end{equation*}
Then $v_i$ ($i=1,2$) satisfy the following equation
\begin{equation}
\label{5.9}
\Delta v_i+e^{v_i}=g(u_i),~i=1,2,
\end{equation}
where $\int_{\Omega}e^{v_1}=\int_{\Omega}e^{v_2}=\lambda.$ It is not difficult to see that $v_1\neq v_2.$ Indeed, if $v_1\equiv v_2$, then we get $g(u_1)=g(u_2)$ by \eqref{5.9}. Combined with the fact $g(u)$ is strictly increasing, we get $u_1=u_2$ and contradiction arises. Thus, $v_1\neq v_2.$

Set
\begin{equation*}
\label{5.10}
t_{i}=\inf_{\Omega}(v_2-v_1)\quad \mathrm{and}\quad t_{s}=\sup_{\Omega}(v_2-v_1).
\end{equation*}
Then we claim that
\begin{equation}
\label{5.11}
\int_{\{v_2-v_1>t\}}(g(u_2)-g(u_1))dx\geq0,~\forall t\in[t_{i},t_{s}].
\end{equation}
We notice that
\begin{equation*}
\int_{\{v_2-v_1>t\}}(g(u_2)-g(u_1))dx=\int_{\{u_2-u_1>t^*\}}(g(u_2)-g(u_1))dx,
\end{equation*}
where
$$t^*=t+\log\int_{\Omega}e^{u_2}dx-\log\int_{\Omega}e^{u_1}dx.$$
If $t^*\geq0$, then there is nothing to prove, while if $t^*<0$, we have
\begin{equation*}
\int_{\{u_2-u_1>t^*\}}(g(u_2)-g(u_1))dx=\int_{\Omega}(g(u_2)-g(u_1))dx-\int_{\{u_2-u_1\leq t^*\}}(g(u_2)-g(u_1))dx\geq0.
\end{equation*}
Thus we have proved the claim.

Next we divide our arguments into two steps.
\medskip

\noindent Step 1. We prove that $\lambda\geq8\pi.$  Suppose $\lambda<8\pi$, we choose $\theta>0$ and $R\in(0,\infty)$ such that
\begin{equation*}
\label{5.12}
\int_{\Omega}e^{v_1}dx=\int_{B_R}e^{U_{\theta}}dx,
\end{equation*}
and let $\phi$ be the symmetrization of $v_2-v_1$ with respect to the measure $e^{v_1}dx$ and $e^{U_{\theta}}dx$. Then it follows from Proposition 5.B and Fubini's theorem that
\begin{equation*}
\label{5.13}
\begin{aligned}
\int_{\{\phi=t\}}|\nabla\phi|d\sigma\leq&\int_{\{v_2-v_1=t\}}|\nabla(v_2-v_1)|d\sigma=-\int_{\{v_2-v_1>t\}}\Delta (v_2-v_1)dx\\
=&\int_{\{v_2-v_1>t\}}(e^{v_2}-e^{v_1})dx-\int_{\{v_2-v_1>t\}}(g(u_2)-g(u_1))dx\\
\leq&\int_{\{v_2-v_1>t\}}(e^{v_2}-e^{v_1})dx=\int_{\{\phi>t\}}e^{U_\theta+\phi}dx
-\int_{\{\phi>t\}}e^{U_{\theta}}dx\\
=&\int_{\{\phi>t\}}e^{U_\theta+\phi}dx-\int_{\{\phi=t\}}|\nabla U_{\theta}|d\sigma,
\end{aligned}
\end{equation*}
for $t\in(t_b,t_s],$ where $t_b=(v_2-v_1)\mid_{\partial\Omega}$ and we used \eqref{5.11}. If $t< t_b$, we find that
$$\partial\{v_2-v_1>t\}=\{v_2-v_1=t\}\cup\partial\Omega.$$
On the other hand, we notice that $v_2-v_1\leq t$ for $t<t_b$ implies that $u_2-u_1\leq 0$, so $g(u_2)-g(u_1)\leq 0$ in $\{v_2-v_1\leq t\}$ for $t<t_b$. Then
\begin{equation*}
\label{5.13b}
\begin{aligned}
\int_{\{\phi=t\}}|\nabla\phi|d\sigma\leq&\int_{\{v_2-v_1=t\}}|\nabla(v_2-v_1)|d\sigma\\
=&-\int_{\{v_2-v_1>t\}}\Delta (v_2-v_1)dx+\int_{\partial\Omega}\partial_{\nu}(v_2-v_1)d\sigma\\
=&\int_{\{v_2-v_1>t\}}(e^{v_2}-e^{v_1})dx+\int_{\{v_2-v_1\leq t\}}(g(u_2)-g(u_1))dx\\
\leq&\int_{\{v_2-v_1>t\}}(e^{v_2}-e^{v_1})dx=\int_{\{\phi>t\}}e^{U_\theta+\phi}dx
-\int_{\{\phi>t\}}e^{U_{\theta}}dx\\
=&\int_{\{\phi>t\}}e^{U_\theta+\phi}dx-\int_{\{\phi=t\}}|\nabla U_{\theta}|d\sigma,
\end{aligned}
\end{equation*}
for $t\in[t_{i},t_{b}),$ where we used
\begin{equation*}
\begin{aligned}
\int_{\partial\Omega}\partial_{\nu}(v_2-v_1)d\sigma&=\int_{\Omega}\Delta(v_2-v_1)dx
=\int_{\Omega}(g(u_2)-g(u_1))dx-\int_{\Omega}(e^{v_2}-e^{v_1})dx\\
&=\int_{\Omega}(g(u_2)-g(u_1))dx.
\end{aligned}
\end{equation*}
Hence
\begin{equation*}
\label{5.14}
\int_{\{\phi=t\}}|\nabla(U_{\theta}+\phi)|d\sigma\leq\int_{\{\phi>t\}}e^{U_{\theta+\phi}}dx
\end{equation*}
for all a.e. $t>t_{i}$. Since $\phi$ is decreasing in $r$, $\psi:=U_{\theta}+\phi$ is strictly decreasing function, and
\begin{equation*}
\label{5.15}
\int_{\partial B_r}|\nabla\psi|d\sigma\leq\int_{B_r}e^{\psi}dx,~\mbox{a.e.}~r\in(0,R).
\end{equation*}
Since $v_1\neq v_2$ and $\int_{\Omega}e^{v_1}dx=\int_{\Omega}e^{v_2}dx$, then $v_2<v_1$ on a subset of $\Omega$ with positive measure. Hence $\phi(R)<0$, therefore
\begin{equation*}
\label{5.16}
\psi(R)=U_{\theta}(R)+\phi(R)<U_{\theta}(R).
\end{equation*}
It is a contradiction by Lemma 5.E., and therefore $\lambda\geq 8\pi.$
\medskip

\noindent Step 2. We prove $\lambda\neq 8\pi$. Suppose $\lambda=8\pi$ and let $\theta_1>0$. From the above arguments we can find there exists $\psi=U_{\theta_1}+\phi\in C^{0,1}(\mathbb{R}^2)$ such that
\begin{equation*}
\label{5.17}
\int_{\Omega}e^{v_1}dx=\int_{\mathbb{R}^2}e^{U_{\theta_1}}dx
=8\pi=\int_{\Omega}e^{v_2}dx=\int_{\mathbb{R}^2}e^{\psi}dx,
\end{equation*}
and
\begin{equation*}
\label{5.18}
\int_{\partial B_r}|\nabla\psi|d\sigma\leq\int_{B_r}e^{\psi}dx~\mbox{for a.e.}~r\in(0,\infty).
\end{equation*}
Since $\int_{\mathbb{R}^2}e^{\psi}dx=\int_{\mathbb{R}^2}e^{U_{\theta_1}}dx$, there exists $r_0\in(0,\infty)$ such that $\psi(r_0)=U_{\theta_1}(r_0)$. Let $\theta_2=\frac{8}{r_0^2\theta_1}$, then it is easy to see that
\begin{equation*}
\label{5.19}
U_{\theta_2}(r_0)=U_{\theta_1}(r_0)=\psi(r_0)
\end{equation*}
by the expression of $U_{\theta}$ in \eqref{5.3}. We notice that $\psi>U_{\theta_1}$ in $B_{r_0}$, it follows from Lemma 5.C that $\theta_1<\theta_2$ and
\begin{equation*}
\int_{B_{r_0}}e^{\psi}dx\geq\int_{B_{r_0}}e^{U_{\theta_2}}dx.
\end{equation*}
While in $\mathbb{R}^2\setminus B_{r_0},$ we have $\psi<U_{\theta_1}$ and it follows from Lemma 5.D that
\begin{equation*}
\int_{\mathbb{R}^2\setminus B_{r_0}}e^{\psi}dx\geq \int_{\mathbb{R}^2\setminus B_{r_0}}e^{U_{\theta_2}}dx.
\end{equation*}
Hence
\begin{equation}
\label{5.20}
8\pi=\lambda=\int_{\mathbb{R}^2}e^{\psi}dx\geq\int_{\mathbb{R}^2}e^{U_{\theta_2}}dx=8\pi.
\end{equation}
Since the solution $u$ of \eqref{5.1} is positive, we have $\Delta v_i+e^{v_i}>0,~i=1,2$, as a consequence, the inequality in \eqref{5.20} is strict by Lemma 5.C and Lemma 5.D. Thus $\lambda\neq8\pi$ and it finishes the proof. \hfill $\square$
\medskip

Before proving Theorem \ref{th.dirichlet}, we shall derive the degree counting formula for \eqref{1.dirichlet} in $H_0^1(\Omega)$ \footnote{$H^1_0(\Omega)=\{u\in H^1(\Omega):u=0~\mathrm{on}~\partial\Omega\}$.}.
\medskip

For any solution of \eqref{1.dirichlet} in $H_0^1(\Omega)$, an inequality of Brezis-Strauss \cite{bs} asserts that (also see \cite[Lemma 2.3]{djlw})
$$\|u\|_{W^{1,q}(\Omega)}\leq C~\mathrm{for}~1<q<2~\mathrm{with~a~constant}~C=C_q>0.$$
We decompose $u=w+v$, where $w,v$ satisfy
\begin{equation*}
\label{5.22}
\Delta w-\beta u=0~\mathrm{in}~\Omega,\quad w=0~\mathrm{on}~\partial\Omega,
\end{equation*}
and
\begin{equation}
\label{5.23}
\Delta v+\lambda\frac{he^v}{\int_{\Omega}he^vdx}=0~\mathrm{in}~\Omega,\quad v=0~\mathrm{on}~\partial\Omega,
\end{equation}
respectively, where $h=e^{w}$. Using the standard elliptic estimate and the fact $u\in W^{1,q}(\Omega)$, we get $w\in W^{3,q}(\Omega)$ and $\|w\|_{C^1(\Omega)}\leq C$. We denote $H=\lambda\frac{h}{\int_{\Omega}he^vdx}$ and it is easy to see that $\|H\|_{C^1(\Omega)}\leq C$. If there is a sequence of solutions $\{u_n\}$ of \eqref{1.dirichlet} with $\lambda=\lambda_n$ such that $\max_{\Omega}u_n\to+\infty$ as $n\to+\infty$, then for equation \eqref{5.23} with $\lambda$ replaced by $\lambda_n$, we get $\max_{\Omega}v_n\to+\infty$ as $n\to+\infty$. By \cite[Theorem 1]{mw}, we have $\lambda_n\to 8m\pi$ for some $m\in\mathbb{N}.$ As a consequence, if $\lambda\neq 8m\pi$, then any solution $v$ of \eqref{5.23} is uniformly bounded, and we can find a constant $C>0$ such that $\|u\|_{L^{\infty}(\Omega)}\leq C$ for any solution of \eqref{1.dirichlet}. Set $T({\lambda})$ to be
\begin{equation*}
\label{5.24}
T(\lambda)=\Delta^{-1}\left(\lambda \frac{e^u}{\int_{\Omega}e^udx}-\beta u\right),
\end{equation*}
which acts in $H^1(\Omega).$ Then the Leray-Schauder degree
\begin{equation*}
d_{\lambda}:=\mathrm{deg}(I+T(\lambda),B_{R},0)
\end{equation*}
is well-defined for $\lambda\neq 8m\pi$ and
$$B_{R}=\{u\in H_0^1(\Omega):\|u\|_{H^1}\leq R\}~\mathrm{with}~R~\mbox{is sufficiently large}.$$
Next, we introduce a homotopy deformation for \eqref{1.dirichlet}
\begin{equation}
\label{5.25}
\begin{cases}
\Delta u-t\beta u+\lambda\frac{e^u}{\int_{\Omega}e^udx}=0~&\mathrm{in}~\Omega,\\
u=0~&\mathrm{on}~\partial\Omega.
\end{cases}
\end{equation}
When $t=1$, equation \eqref{5.25} is \eqref{1.dirichlet}, while if $t=0$, equation \eqref{5.25} is the mean field equation in bounded domain. It is not difficult to see that as $t$ changes from $1$ to $0$, we can argue as above to get that the solution $u_t$ of \eqref{5.25} is uniformly bounded provided $\lambda\neq 8m\pi.$ Therefore the Leray-Schauder degree of the above equation is the same for $t=1$ and $t=0$. Together with the degree formula for the mean field equation in bounded domain (see \cite[Theorem 1.3]{cl}), we get

\begin{theorem}
\label{th5.2}
Let $\Omega$ be a non-empty bounded open set and $8m\pi<\lambda <8(m+1)\pi$ for some positive integer $m.$ Then the Leray-Schauder degree of \eqref{1.dirichlet} $d_{\lambda}=\left(\begin{matrix} m+g-1\\ m\end{matrix}\right)$, where $g$ denotes the number of holes in $\Omega.$ While if $0\leq\lambda<8\pi$, $d_{\lambda}=1.$
\end{theorem}

Next we prove Theorem \ref{th.dirichlet} by Theorems \ref{th5.1} and \ref{th5.2}.
\medskip

\noindent{\em Proof of Theorem \ref{th.dirichlet}.} First we claim that all the solutions to the equation \eqref{1.dirichlet} are positive. Indeed, if $u$ is not positive in $\Omega$, letting $p$ be the point in $\Omega$ where $u$ obtains its minimal value, then we have $u(p)\leq 0$ and $\Delta u(p)\geq0$.  Therefore
\begin{equation*}
\Delta u-\beta u+\lambda\frac{e^u}{\int_{\Omega}e^udx}>0~\mathrm{at}~p,
\end{equation*}
which gives a contradiction. Hence, the claim is true. Next, it is easy to see that $g(u)=u$ satisfies the condition $g(x),g'(x)>0$ for $x>0.$ Then Theorem \ref{th.dirichlet} follows by Theorems \ref{th5.1} and \ref{th5.2}. Thus we finish the proof. \hfill $\square$
\medskip

\section{Appendix}

\begin{lemma}
\label{le6.1}
Let $u$ be a solution to the boundary value problem (\ref{1.6}) with $\beta>0$. Then
\begin{equation*}
\label{a.2}
u>0~\mbox{in}~\Omega.
\end{equation*}
\end{lemma}

\begin{proof}
Let
$$a=\min_{x\in\overline\Omega}u(x).$$
We shall prove $a\geq0$ by contradiction. Suppose $a<0$, we first prove the value $a$ can not be obtained by $u$ in $\Omega$,
indeed if
$u(x_0)=a$ with $x_0\in\Omega$, we get
\begin{equation*}
\label{a.4}
\Delta u(x_0)-\beta u(x_0)+\lambda\frac{e^{u(x_0)}}{\int_{\Omega}e^{u}dx}>0,
\end{equation*}
contradiction arises. Hence, $a$ can only be obtained by $u$ on $\partial\Omega$, say at $x_0\in\partial\Omega.$ By the continuity of $u$, we
can always find $\Omega'\subset\Omega$ with $x_0\in\overline\Omega'\cap\overline\Omega$ and
$u(x)\leq0$ in $\Omega'$. As a consequence, we can see that
\begin{equation*}
\label{a.5}
\Delta u=\beta u-\lambda\frac{e^{u}}{\int_{\Omega}e^{u}dx}\leq0~\mbox{in}~\Omega'.
\end{equation*}
By the Hopf boundary lemma, we get $\frac{\partial u}{\partial\nu}(x_0)<0$, which contradicts to the Neumann boundary condition. Thus the
assumption is not true and $a\geq0$. On the other hand, by the above arguments we can also show that $u(x)$ can not reach the value $a$ in $\Omega$ if $a=0$. Therefore $u(x)>0$ in $\Omega$ and the lemma is proved.
\end{proof}
\bigskip

\noindent \textbf{Acknowledgement:} Part of this work was finished when the third author was a postdoctoral fellow in the Department of Applied Mathematics in  the Hong Kong Polytechnic University supported by the project G-YBKT. The research of J. Wang is supported by NSFC of China (Grants 11571140, 11671077), Fellowship of Outstanding Young Scholars of Jiangsu Province (BK20160063). the Six big talent peaks project in Jiangsu Province(XYDXX-015), and NSF of Jiangsu Province (BK20150478).
The research of Z.A. Wang is supported by the Hong Kong GRF grant PolyU 153041/15P. The research of W. Yang  is supported by CAS Pioneer Hundred Talents Program (Y8S3011001).

\end{document}